\documentclass[a4paper,10pt,twoside]{article}

\usepackage[left=1.4in,right=1.4in,top=1.6in,bottom=1.7in]{geometry}
\usepackage{lmodern}
\usepackage[utf8]{inputenc}
\usepackage[T1]{fontenc}
\usepackage[english]{babel}
\usepackage{mathtools,amsmath,amssymb,amsfonts,amsthm,mathrsfs}
\usepackage{bbm}
\mathtoolsset{showonlyrefs}
\usepackage{pgfplots}
\usepackage{graphicx}
\usepackage{caption}
\usepackage{subcaption}

\usepackage{booktabs}
\usepackage{multirow}
\usepackage{rotating,tabularx}
\usepackage{adjustbox}
\usepackage{enumitem}

\newtheorem{theorem}{Theorem}[section]
\newtheorem{definition}{Definition}[section]
\newtheorem{lemma}[theorem]{Lemma}
\newtheorem{proposition}[theorem]{Proposition}

\newtheorem{conjecture}[theorem]{Conjecture}
\theoremstyle{definition}

\theoremstyle{definition}

\numberwithin{equation}{section}

\usepackage[hidelinks]{hyperref}
\usepackage{fancyhdr}
\makeatletter
\def\thanks#1{\protected@xdef\@thanks{\@thanks
        \protect\footnotetext{#1}}}
\makeatother
\renewenvironment{abstract}[1]
{\list{}{\setlength{\leftmargin}{3em}
 \setlength{\rightmargin}{\leftmargin}}\item[]
\textbf{\abstractname.} #1\relax}
{\endlist}
\title{\textbf{Chaos in periodically forced\\reversible vector fields}}

\author{Isabel S. Labouriau and Elisa Sovrano}
\date{}

\begin{document}

\newcommand{\Addresses}{\begin{minipage}[l]{21cm}
\small
\noindent Isabel S. Labouriau\\
Centro de Matem\'atica da Universidade do Porto\\
Rua do Campo Alegre 687, 4169-007 Porto, Portugal\\
email: \href{mailto:islabour@fc.up.pt}{\texttt{islabour@fc.up.pt}}

\medskip

\small
\noindent Elisa Sovrano\\
Istituto Nazionale di Alta Matematica ``Francesco Severi''\\
c/o Dipartimento di Matematica e Geoscienze, 
\\Universit\`a degli Studi di Trieste,\\
Via A. Valerio 12/1, 34127 Trieste, Italy\\
email: \href{mailto:esovrano@units.it}{\texttt{esovrano@units.it}}
\end{minipage}}

\maketitle
\thispagestyle{empty}


\begin{abstract}
\noindent
We discuss the appearance of chaos in  time-periodic perturbations of reversible vector fields in the plane. We use the normal forms of codimension~$1$ reversible vector fields and discuss the ways a time-dependent periodic forcing term of pulse form may be added to them to yield topological chaotic behaviour. Chaos here means that the resulting dynamics is semiconjugate to a shift in a finite alphabet. 
The results rely on the classification of reversible vector fields and on the theory of topological horseshoes. This work is part of a project of studying periodic forcing of symmetric vector fields.


\end{abstract}

\section{Introduction}\label{section-1}
A standard classification of continuous dynamical systems defined by a set of first order ordinary differential equations distinguishes between conservative systems and dissipative ones~\cite{Ru-81}. On the one hand, conservative systems can be described by a Hamiltonian function. By varying the initial conditions, these systems can exhibit regions of regular motions surrounded by a sea of chaotic ones. Instead, dealing with dissipative systems, conserved quantities are no longer guaranteed, and chaotic regions could coexist with stable equilibria, limit cycles, and strange attractors.

In between conservative and dissipative systems, there are systems with reversing symmetries. By reversible dynamical systems we mean those admitting an involution in phase space which reverses the direction of time (see~\cite{De-76,LaRo-98,Se-86,Te-97}). It is shown that these systems despite having similar features to Hamiltonian ones (e.g., at an elliptic equilibrium can possess the same structure), yet they are different because they can also have attractors and repellers. The additional structure given by reversing symmetries allows exhibiting complex behaviors for codimension one bifurcations, and so, it can be responsible for chaotic dynamics. 

The goal of this paper is to find chaos for a class of planar periodically perturbed reversible systems whose normal form analysis is studied in~\cite{Te-97}. We take into account the local bifurcations of low codimension by arguing what dynamical behaviors we can expect. Our main result is the following.

\begin{conjecture}\label{th-intro}
Let $X_{\lambda}(x,y)$ be a fixed type of normal form for a one-parameter family of  codimension 1 reversible vector fields. 
Let $\lambda_1$ and $\lambda_2$ be two real distinct values. Suppose that the dynamical system $\dot{X}=X(x,y)$ switches in a $T$-periodic manner between 
\begin{equation}
\dot{X}=X_{\lambda_1}(x,y) \text{ for $t\in[0,\tau_1)$}\quad \text{and}\quad \dot{X}=X_{\lambda_2}(x,y) \text{ for $t\in[\tau_1,\tau_1+\tau_2)$}
\end{equation}
with $\tau_1+\tau_2=T$. 
Then for open sets of the parameters $(\lambda_1,\lambda_2)$ and for $\tau_1$ and $\tau_2$ in open intervals
 there exist infinitely many $T$-periodic solutions as well as chaotic-like dynamics for the problem $\dot{X}=X(x,y)$.
\end{conjecture}

The paper is organized as follows. In Section~\ref{section-2} we discuss the classification of plane reversible vector fields of codimension $0$ and $1$. In Section~\ref{section-3} we give a review of the concept of symbolic dynamics and topological horseshoes. We collect preliminary topological results in the phase-plane that can produce chaotic dynamics. In Section~\ref{section-4} we prove Conjecture~\ref{th-intro} for the two of the four normal forms of  codimension 1 reversible vector fields: $i)$ saddle type and  $ii)$ cusp type. We conjecture that the other two possible normal forms, namely $iii)$  nodal type and $iv)$  focal type, may also  be amenable to the same treatment.

\section{Planar reversible systems}\label{section-2} 
In \cite{Te-97}, M.~A.~Teixeira has provided a local classification of 2D reversible systems of codimension less than or equal to two. 
A dynamical system $\dot{X}=V(X)$ is called {\em reversible} if there is a phase space {\em involution} $h$ (i.e.,~$h^2=\mathrm{Id}$) such that $Dh(p)V(p)=-V(h(p))$ for $p\in\mathbb{R}^2$.
We deal with reversible planar systems where the involution is $h(x,y)=(x,-y)$. Hence, we consider a dynamical system of the following form 
\begin{equation}\label{eq-2.1}
\begin{cases}
\dot{x}= y f(x,y^2), \\
\dot{y}= g(x,y^2),
\end{cases}
\end{equation}
where the functions $f$ and $g$ are smooth. We consider the behaviour of \eqref{eq-2.1} near the origin, often making the assumption that it has an equilibrium at the origin.
In the half-plane $y>0$, by using the transformation $u=x$ and $v=y^2$, we can write system \eqref{eq-2.1} equivalently as follows 
\begin{equation}\label{eq-2.2}
\begin{cases}
\dot{u}= \sqrt{v}\, f(u,v), \\
\dot{v}= 2\sqrt{v}\, g(u,v).
\end{cases}
\end{equation}
Through the symmetry properties of the vector field $X(x,y)$ associated with \eqref{eq-2.1}, the behavior of $X$ near $(0,0)$ may
be described by the analysis in the half-plane
$\{(u,v)\in\mathbb{R}^2\colon v\geq 0\}$ of the vector field $Y(u,v)=\left(f(u,v),g(u,v)\right)$.  

\subsection{Normal forms}
Following the work in \cite{Te-97}, the generic equilibria of  reversible ODEs near the origin are either
centers and saddles on the line of symmetry or a couple of repellers and attractors, as in Figure~\ref{fig-1}.

\begin{figure}[h]
\centering
\includegraphics[height=3.5cm]{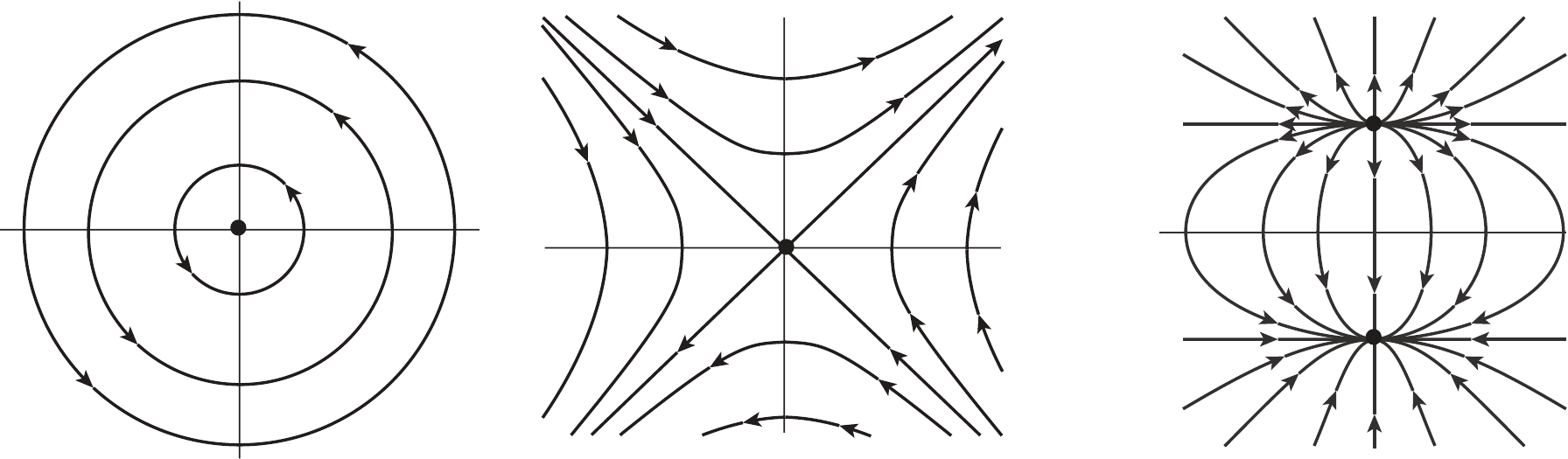} 
\caption{Phase-portraits of equilibria occurring in generic 2D~reversible fields. The local geometry may be of a center~(left), a saddle~(middle), or a pair of attractor and repeller~(right).}\label{fig-1}
\end{figure}
 
Let $S$ be the line $\{(x,0)\colon x\in\mathbb{R}\}$, the set of fixed points for $h$. An equilibrium point of V that lies on $S$ is called a {\em symmetric equilibrium}. 

\begin{theorem}[\cite{Te-97}]\label{normal-0}
The normal forms around a symmetric equilibrium at $(0,0)$ of a structurally stable reversible vector field $X$ are: 
\begin{itemize}
\item $X(x,y)=(y,x)$,
\item $X(x,y)=(y,-x)$.
\end{itemize}
\end{theorem}
In the first case the origin is a center, and in the second one it is a saddle.
The next result classifies one parameter families $X_{\lambda}$ of reversible vector fields such that $X_0$ has a symmetric equilibrium at the origin. 
\begin{theorem}[\cite{Te-97}]\label{normal-1}
The normal forms of one-parameter families of structurally stable reversible vector fields $X_{\lambda}$ near a symmetric equilibrium at $(0,0)$ are:
\begin{enumerate}[leftmargin=*,label=$\roman*)$]
\item saddle type: $X_\lambda(x,y)=(xy,x-y^2+\lambda)$,
\item cusp type: $X_\lambda(x,y)=(y,x^2+\lambda)$,
\item nodal type: $X_\lambda(x,y)=(xy,x+2y^2+\lambda)$ or $X_\lambda(x,y)=(-xy,x-2y^2+\lambda)$,
\item focal type: $X_\lambda(x,y)=(xy +y^3, -x+ y^2+\lambda)$.
\end{enumerate}
\end{theorem}
Depending on $\lambda$, the phase-portraits of the above normal forms can be described as follows. 

\begin{figure}[htb]
\centering
\includegraphics[height=3.5cm]{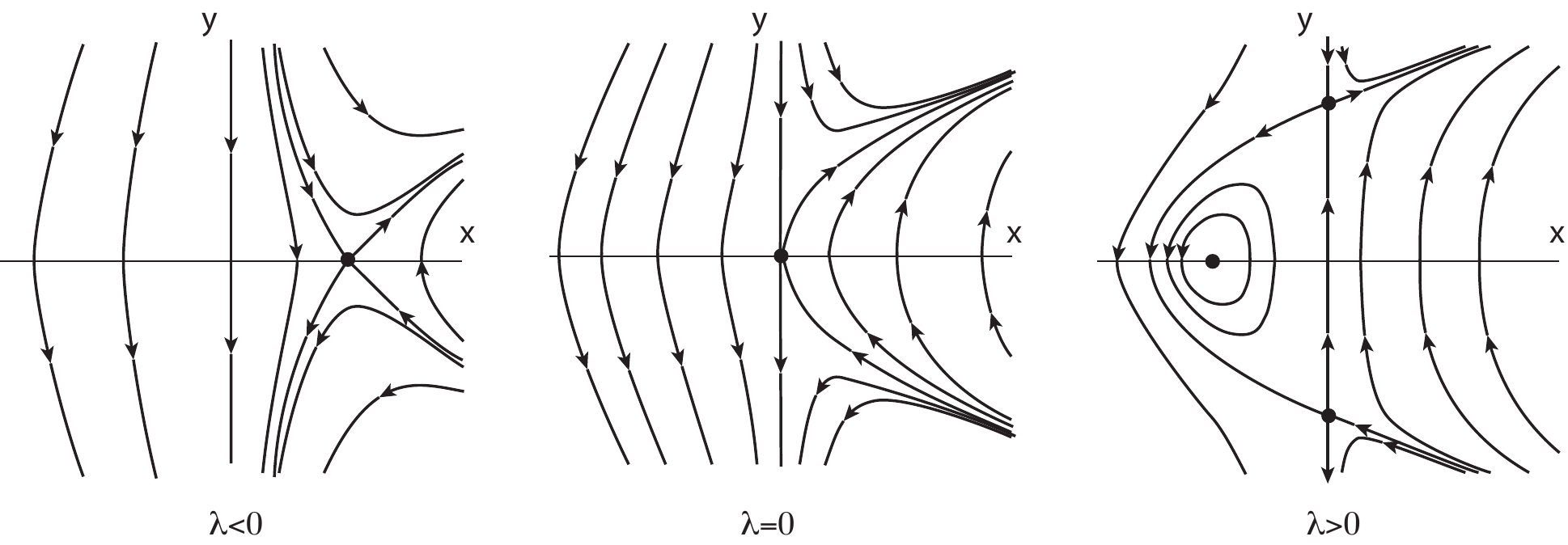} 
\caption{Phase-portraits reversible vector fields of saddle type.
}\label{figSaddle}
\end{figure}

Figure~\ref{figSaddle} shows the phase portraits of the \textit{saddle type}.
When $\lambda\leq0$ there is an equilibrium at $(-\lambda,0)$ which is a saddle. When $\lambda>0$ there are three equilibria: a center and two saddles at $(-\lambda,0)$, $(0,-\sqrt{\lambda})$ and $(0,\sqrt{\lambda})$, respectively. The saddle points are connected through heteroclinic trajectories which surround periodic orbits.

\begin{figure}[htb]
\centering
\includegraphics[height=3.5cm]{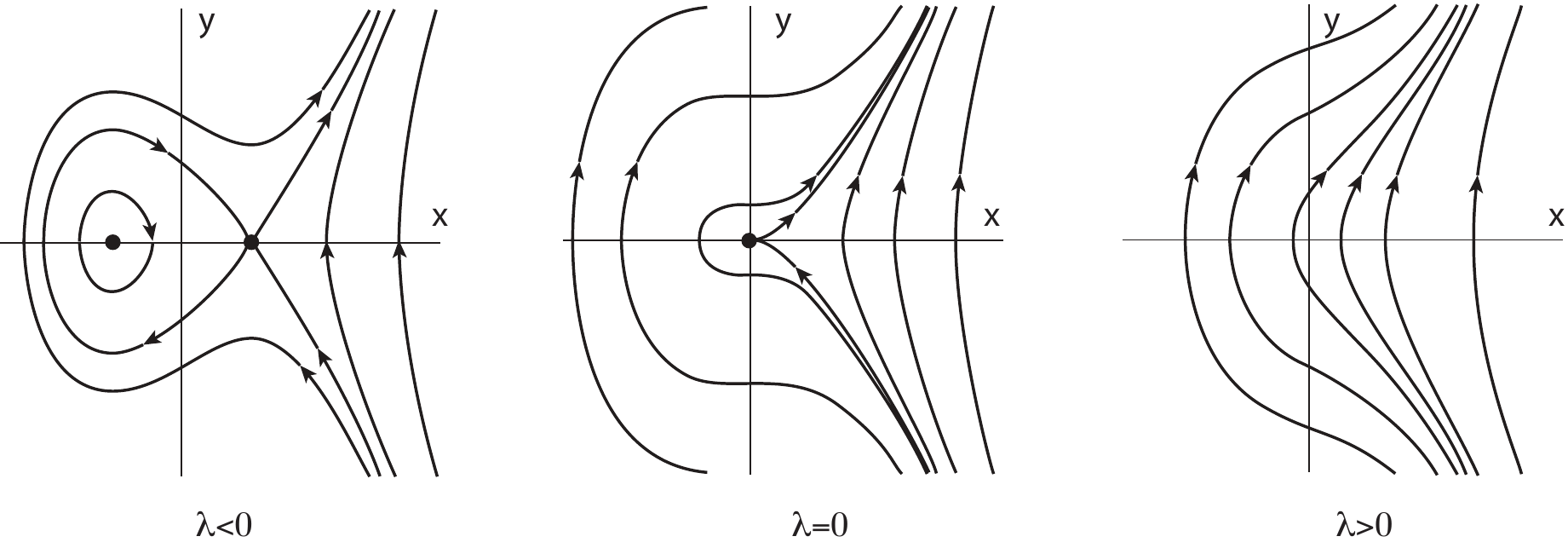} 
\caption{Phase-portraits reversible vector fields of cusp type.
}\label{figCusp}
\end{figure}

Concerning the \textit{cusp type} when $\lambda<0$ there are two equilibria: a center and a saddle which are at $(-\sqrt{-\lambda},0)$ and $(\sqrt{-\lambda},0)$, respectively. Due to the reversibility, the only periodic orbits are the ones that meet the points $(x,0)$ with $-2\sqrt{-\lambda}<x<\sqrt{-\lambda}$, as in Figure~\ref{figCusp}. Moreover, these orbits are located inside the homoclinic trajectory that passes through $(-2\sqrt{-\lambda},0)$. When $\lambda=0$ there is only an equilibrium which is a degenerate saddle at $(0,0)$ and all the orbits are unbounded. When $\lambda>0$ there are no equilibria. 
 
\begin{figure}[htb]
\centering
\includegraphics[height=3.5cm]{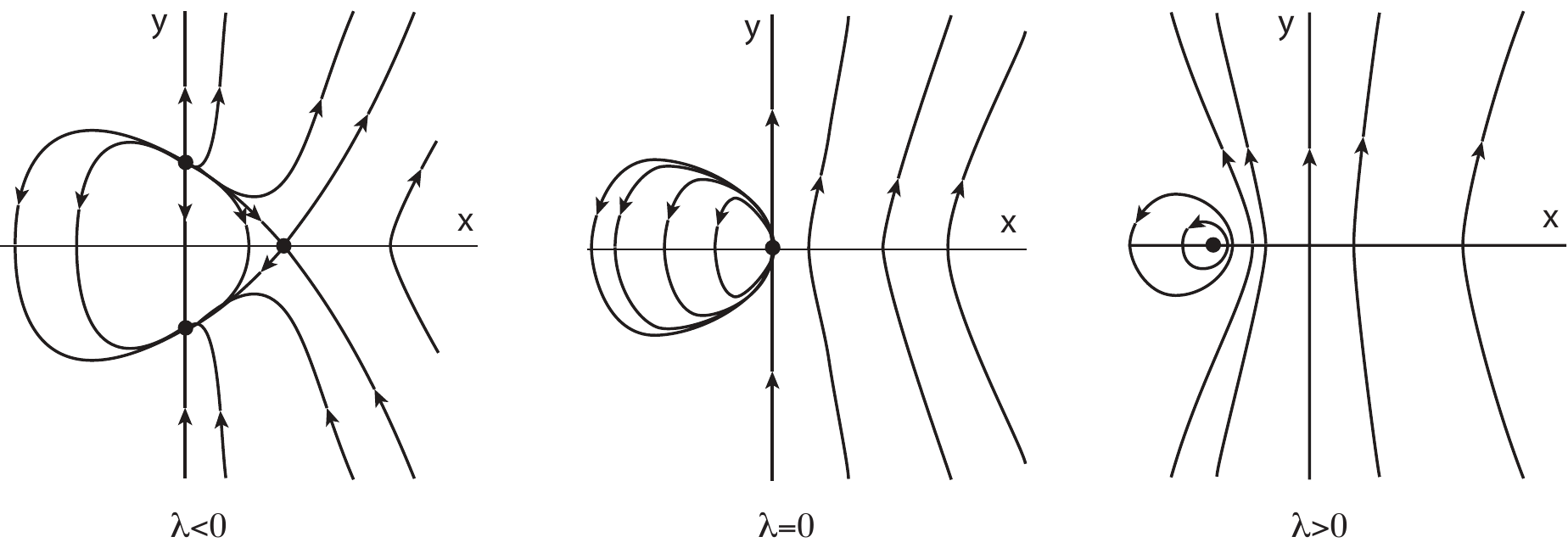} 
\caption{Phase-portraits reversible vector fields of nodal type  (first case).
}\label{figNodal}
\end{figure}

For the \textit{nodal type} (first case, shown in Figure~\ref{figNodal}) when $\lambda<0$ there are three equilibria: an attractor, a repeller and a saddle, located respectively at $(0,-\sqrt{-\lambda/2})$, $(0,\sqrt{-\lambda/2})$ and $(-\lambda,0)$. When $\lambda=0$ there is only an equilibrium at $(0,0)$. When $\lambda>0$ there is only an equilibrium at $(-\lambda,0)$ which is a center and in the half-plane $x<0$ all the orbits are periodic. In the second case there is always an equilibrium at $(-\lambda,0)$ and for  $\lambda>0$ there is also a pair of equilibria at $(0,\pm\sqrt{\lambda/2})$.

\begin{figure}[htb]
\centering
\includegraphics[height=3.5cm]{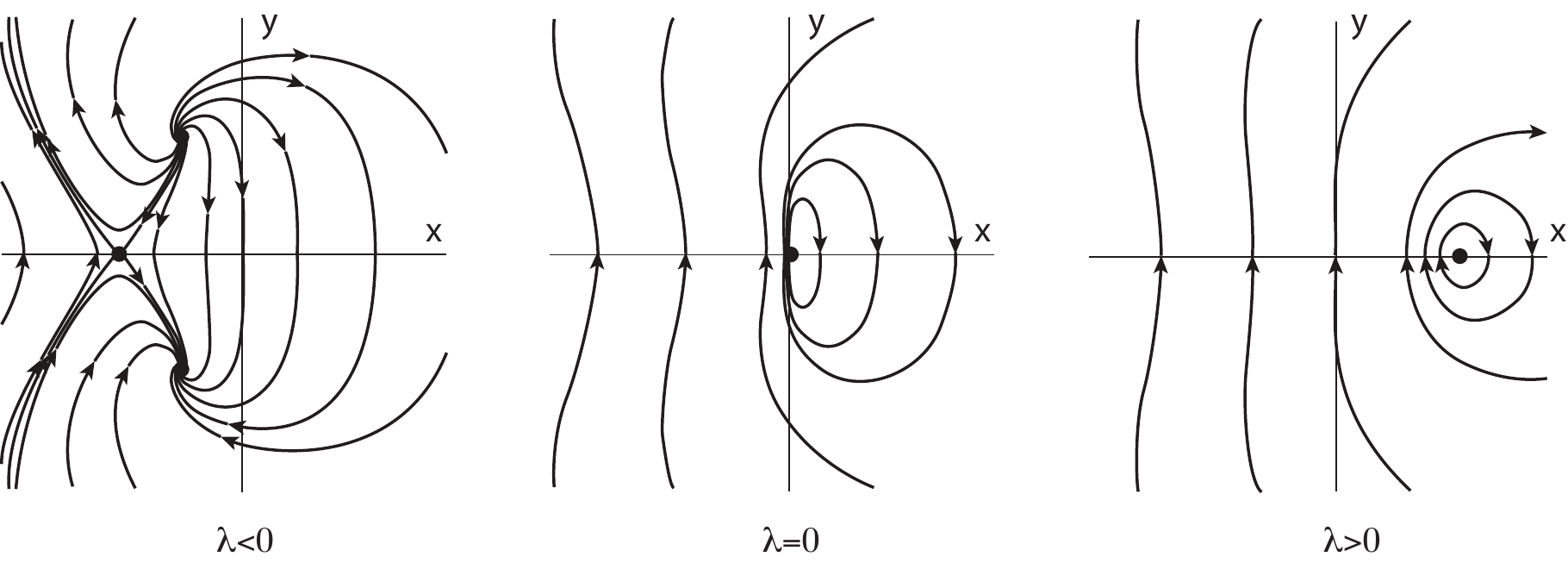} 
\caption{Phase-portraits reversible vector fields of focal type.
}\label{figFocal}
\end{figure}

For the \textit{focal type} when $\lambda<0$ there are three equilibria: a saddle  and two foci at $(\lambda,0)$, $(\lambda/2,-\sqrt{-\lambda/2})$ and $(\lambda/2,\sqrt{-\lambda/2})$, respectively. When $\lambda\geq0$ there is only an equilibrium at $(\lambda,0)$ which is a center and all the orbits are periodic as in Figure~\ref{figFocal}.

\section{Background on chaotic dynamics and preliminary results}\label{section-3}

\subsection{Symbolic dynamics and chaos}\label{subsection-3.1}
To review the topological approach exploited throughout the paper, we start by introducing some notation and definitions of symbolic dynamics. General information on the subject may be found in the book by Guckenheimer and Holmes~\cite{GH}, with examples in Chapter~2 and a more general case in Chapter~5.
A more detailed treatment is given by Wiggins and Ottino~\cite{WiOt-04}. 
The point of view used here is similar to that of Kennedy and Yorke in~\cite{KeYo-01}  of Margheri~\emph{et al} in~\cite{MRZ-10} and of Medio~\emph{et al} in~\cite{MPZ-09}.

Let $\Sigma_m:=\{0,\dots,m-1\}^{\mathbb{Z}}$ be the set of all two-sided sequences $S=(s_i)_{i\in\mathbb{Z}}$ with $s_i\in \{0,\dots,m-1\}$ for each $i\in\mathbb{Z}$ endowed with a standard metric that makes $\Sigma_m$ a compact space with the product topology. We define the shift map $\sigma\colon \Sigma_m \to \Sigma_m$ by $\sigma(S) = S'=(s'_i)_{i\in\mathbb{Z}}$ with $s'_i = s_{i+1}$ for all $i\in {\mathbb Z}.$
We say that  a map $h$ on a metric space is semiconjugate (respectively, conjugate) to the shift map on $m$ symbols if there exists a compact invariant set $\Lambda$ and a continuous and surjective (respectively, bijective) map $\Pi \colon \Lambda \to \Sigma_m$ such that $\Pi\circ h(w)= \sigma\circ\Pi(w),$ for all $w\in \Lambda.$ 

The deterministic chaos is usually associated with the possibility to reproduce all the possible outcomes of a coin-tossing experiment, by varying the initial conditions within the dynamical system. We can express this concept using the symbolic dynamics of the shift map on the sets of two-sided sequences of $2$ symbols. However, by considering a finite alphabet made by $m$ symbols the possible dynamics can be more complex. Hence, in the sequel we adopt the following definition of chaos (cf.,~\cite{MRZ-10,MPZ-09}).

\begin{definition}[Symbolic dynamics]\label{def-symb}
Let $h\colon\mathrm{dom}\,h\subseteq\mathbb{R}^{2}\to\mathbb{R}^{2}$ be a map and
let $\mathcal{D}\subseteq\mathrm{dom}\,h$ be a nonempty set.
We say that \emph{$h$ induces chaotic dynamics on $m\geq 2$ symbols on a set $\mathcal{D}$}
if there exist $m$ nonempty pairwise disjoint compact sets
$\mathcal{K}_{0}, \dots, \mathcal{K}_{m-1}\subseteq\mathcal{D}$
such that for each two-sided sequence $(s_{i})_{i\in\mathbb{Z}}\in\Sigma_{m}$
there exists a corresponding sequence $(w_{i})_{i\in\mathbb{Z}}\in\mathcal{D}^{\mathbb{Z}}$ such that
\begin{equation}\label{eq-ap1}
w_{i}\in\mathcal{K}_{s_{i}} \text{ and } w_{i+1}=h(w_{i}) \text{ for all } i\in\mathbb{Z},
\end{equation}
and, whenever $(s_{i})_{i\in\mathbb{Z}}\in\Sigma_{m}$ is a $k$-periodic sequence for some $k\geq1$ there exists a
$k$-periodic sequence $(w_{i})_{i\in\mathbb{Z}}\in\mathcal{D}^{\mathbb{Z}}$ satisfying \eqref{eq-ap1}.
\end{definition}

For a one-to-one map $h$, Definition~\ref{def-symb} ensures the existence of a nonempty compact invariant set $\Lambda\subseteq\cup^{m-1}_{i=0}\mathcal{K}_{i}\subseteq\mathcal{D}$ and a continuous surjection $\Pi$ such that $h_{|\Lambda}$ is semiconjugate to the Bernoulli shift map on $m\geq 2$ symbols.
Moreover, it guarantees that the set of the periodic points of $h$ is dense in $\Lambda$ and, for all two-sided periodic 
sequences $S\in\Sigma_{m}$, the preimage $\Pi^{-1}(S)$ contains a periodic point of $h$ with the same period (cf.~\cite[Th.~2.2]{MPZ-09}). In this respect Definition~\ref{def-symb} is related,  by means of \cite[Th.~2.3]{MPZ-09}, to the concept of topological horseshoe introduced in~\cite{KeYo-01}. This is a weaker notion of chaos than the Smale's horseshoe (see~\cite[ch.~5]{GH}) because the latter requires the full conjugacy between $h_{|\Lambda}$ and the shift map on $m$ symbols.

We introduce the notion of an oriented topological rectangle and the stretching along the path property by borrowing the notations and definitions from~\cite{MRZ-10,PVZ-18}. The pair $\widehat{\mathcal{R}}:=(\mathcal{R},\mathcal{R}^{-})$ is called \emph{oriented topological rectangle}
if $\mathcal{R}\subseteq \mathbb{R}^{2}$ is a set homeomorphic to $[0,1]\times[0,1]$, and $\mathcal{R}^{-}=\mathcal{R}^{-}_{l}\cup\mathcal{R}^{-}_{r}$, where $\mathcal{R}^{-}_{l}$ and $\mathcal{R}^{-}_{r}$ are two disjoint compact arcs contained in $\partial\mathcal{R}.$

\begin{definition}[SAP property]\label{def-stretching}
Given two topological oriented rectangles $\widehat{\mathcal{R}}_1:=(\mathcal{R}_1,\mathcal{R}_1^{-})$,
$\widehat{\mathcal{R}}_2:=(\mathcal{R}_2,\mathcal{R}_2^{-})$ and a continuous map $h:\mathrm{dom}\,h\subseteq\mathbb{R}^{2}\to\mathbb{R}^{2}$,
we say that $h$ \emph{stretches $\widehat{\mathcal{R}}_1$ to $\widehat{\mathcal{R}}_2$ along the paths}
if there exists a compact subset $\mathcal{K}$ of $\mathcal{R}_1\,\cap\mathrm{dom}\,h$
and for each path $\gamma\colon[0,1]\to\mathcal{R}_1$ such that $\gamma(0)\in\mathcal{R}^{-}_{1,l}$
and $\gamma(1)\in\mathcal{R}^{-}_{1,r}$ (or vice-versa), there exists $[t_0,t_1]\subseteq[0,1]$ such that
\begin{itemize}
\item $\gamma(t)\in \mathcal{K}$ for all $t\in[t_0,t_1]$,
\item $h(\gamma(t))\in\mathcal{R}_2$ for all $t\in[t_0,t_1]$,
\item $h(\gamma(t_0))$ and $h(\gamma(t_1))$ belong to different components of $\mathcal{R}_2^{-}$.
\end{itemize}
In this case, we write
\[(\mathcal{K},h)\colon\widehat{\mathcal{R}}_1\mathrel{\Bumpeq\!\!\!\!\!\!\longrightarrow}\widehat{\mathcal{R}}_2.\]
\end{definition}
Given a positive integer $m$, we say that $h$ \emph{stretches $\widehat{\mathcal{R}}_1$ to $\widehat{\mathcal{R}}_2$ along the paths with crossing number} $m$
and we write
\[h\colon\widehat{\mathcal{R}}_1\mathrel{\Bumpeq\!\!\!\!\!\!\longrightarrow^{m}}\widehat{\mathcal{R}}_2\]
if there exist $m$ pairwise disjoint compact sets $\mathcal{K}_{0},\dots,\mathcal{K}_{m-1}\subseteq \mathcal{R}_1\cap\mathrm{dom}\,h$
such that $(\mathcal{K}_{i},h)\colon\widehat{\mathcal{R}}_1\mathrel{\Bumpeq\!\!\!\!\!\!\longrightarrow}\widehat{\mathcal{R}}_2$ for each $i\in\{0,\dots,m-1\}$.

Finally, in order to detect chaos, a useful topological tool is the Stretching Along the Paths (SAP) method introduced in~\cite{MPZ-09}. In our framework, it can be stated as follows (cf.,~\cite[Th.~2.1]{MRZ-10}).
\begin{theorem}[SAP method]\label{th-app1}
Let $h_1\colon\mathrm{dom}\,\nu\subseteq\mathbb{R}^{2}\to\mathbb{R}^{2}$ and
$h_2\colon\mathrm{dom}\,\eta\subseteq\mathbb{R}^{2}\to\mathbb{R}^{2}$ be continuous maps.
Let $\widehat{\mathcal{R}}_1=(\mathcal{R}_1,\mathcal{R}_1^{-})$ and $\widehat{\mathcal{R}}_2=(\mathcal{R}_2,\mathcal{R}_2^{-})$
be two oriented rectangles in $\mathbb{R}^{2}$. Suppose that
\begin{itemize}
\item there exist $n\geq 1$ pairwise disjoint compact subsets of $\mathcal{R}_1\,\cap\,\mathrm{dom}\,\nu$,
$\mathcal{Q}_{0},$ $\dots,$ $\mathcal{Q}_{n-1}$, such that
$(\mathcal{Q}_{i},h_1)\colon\widehat{\mathcal{R}}_1\mathrel{\Bumpeq\!\!\!\!\!\!\longrightarrow}\widehat{\mathcal{R}}_2$ for $i=0,\dots,n-1$,
\item there exist $m\geq 1$ pairwise disjoint compact subsets of $\mathcal{R}_2\,\cap\,\mathrm{dom}\,\eta$, 
$\mathcal{K}_{0},$ $\dots,$ $\mathcal{K}_{m-1}$, such that
$(\mathcal{K}_{i},h_2)\colon\widehat{\mathcal{R}}_2\mathrel{\Bumpeq\!\!\!\!\!\!\longrightarrow}\widehat{\mathcal{R}}_1$ for $i=0,\dots,m-1$.
\end{itemize}
If at least one between $n$ and $m$ is greater than or equal to $2$, then the map $h=h_2\circ h_1$
induces chaotic dynamics on $n\times m$ symbols on
\[\mathcal{Q}^{*}=\bigcup_{\substack{i=0,\dots,n-1 \\j=0,\dots,m-1}} \mathcal{Q}_{i}\cap\nu^{-1}(\mathcal{K}_{j}).\]
\end{theorem}
\noindent For the proof of Theorem~\ref{th-app1} we refer to~\cite[Th.~2.1]{MRZ-10}.

\subsection{Topological tools in the phase-plane}\label{subsection-3.2} 
The geometry associated to the phase-portrait of~\eqref{eq-2.1} exhibits unbounded solutions and periodic trajectories. 
These configurations guarantee the existence of two types of invariant regions: topological strips and topological annuli confined between unbounded and bounded solutions, respectively. In this section we will give some preliminary topological results on the phase-plane $(x,y)$ needed to establish the dynamics induced by~\eqref{eq-2.1}. 

By a {\em topological strip}~$\mathcal{S}$ we mean the image of a straight strip of finite width $\mathbf{S}:=\{(x,y)\in\mathbb{R}^2\colon x_1<x<x_2,\, -1\leq y\leq 1\}$ through a locally defined homeomorphism 
\[
h_\mathbf{S}\colon (x_1,x_2)\times[-1,1] \to \mathcal{S}.
\] 
Let a \textit{bridge} in~$\mathcal{S}$ be the image by $h_\mathbf{S}$ of any simple continuous curve $\gamma\colon \mathopen{[}a,b\mathclose{]}\to \mathbf{S}$ such that $\gamma(a)=(\hat{x},-1)$ and $\gamma(b)=(\check{x},1)$ for some $\hat{x},\,\check{x}\in(x_1,x_2)$ or, viceversa,  $\gamma(a)=(\check{x},1)$ and $\gamma(b)=(\hat{x},-1)$.

A {\em topological annulus}~$\mathcal{A}$ is defined as the image of a rectangular region~$\mathbf{A}:=\{(x,y)\in\mathbb{R}^2\colon 1\leq x\leq 2,\, -1\leq y\leq 1\}$ through a continuous map
\[
h_\mathbf{A}\colon [1,2]\times[-1,1] \to \mathcal{A},
\]
such that  the restriction of $h_\mathbf{A}$ to $(1,2)\times[-1,1]$ is a homeomorphism and $h_\mathbf{A}(1,y)=h_\mathbf{A}(2,y)$. We notice that the restriction to $(1,2)\times[-1,1]$ yields a strip. Moreover, the boundary of the topological annulus $\partial \mathcal{A}$ is the union of two Jordan curves $\partial^i \mathcal{A}:=h_\mathbf{A}(x,-1)$ and $\partial^e \mathcal{A}:=h_\mathbf{A}(x,1)$. 
We denote the portion of the plane outside a generic Jordan curve $\Gamma$ by $out(\Gamma)$ and the one inside by $in(\Gamma)$. For identification purposes, let $\partial^i \mathcal{A}\subset in(\partial^ e\mathcal{A})$. In this manner, we can identify two connected sets, one bounded and another one unbounded given by $in(\partial^i \mathcal{A})$ and $out(\partial^e \mathcal{A})$, respectively. Let a \textit{ray} in~$\mathcal{A}$ be any simple continuous curve $\gamma\colon \mathopen{[}a,b\mathclose{]}\to \mathcal{A}$ such that $\gamma(a)\in \partial^i \mathcal{A}$ and $\gamma(b)\in \partial^e \mathcal{A}$ or, viceversa, $\gamma(a)\in \partial^e \mathcal{A}$ and $\gamma(b)\in \partial^i \mathcal{A}$.

We are interested in crossing configurations between either an annulus and a strip or two annuli. In particular we are looking for similarities with the geometry of the linked-twist maps (see~\cite{PPZ-08,WiOt-04}). Hence, we introduce the following definition and in Figure~\ref{fig-link} we provide a visual representation of the linkage condition between an annulus and a strip.

\begin{definition}[Linkage condition]\label{def-link}
Let $\mathcal{A}$ be a topological annulus and $\mathcal{S}$ be a topological strip. We say that $\mathcal{A}$ is linked with $\mathcal{S}$ if there exist a bridge $\gamma_1$ in~$\mathcal{S}$, a ray $\gamma_2$ in~$\mathcal{A}$, and a topological ball $B$ containing $\mathcal{A}$ such that:
\begin{itemize}
\item $\gamma_1\subset in(\partial^i \mathcal{A})$;
\item $\gamma_2\cap \mathcal{S}=\emptyset$;
\item $(\mathcal{S}\setminus\gamma_1)\cap\partial B$ consists of exactly two disjoint bridges.
\end{itemize}
\end{definition}
From Definition~\ref{def-link} we observe that when $\mathcal{A}$ is linked with $\mathcal{S}$, then the topological ball $B$ is cut into two connected components $B^+$ and $B^-$.

\begin{figure}[htb]
\centering
\begin{tikzpicture}[scale=1]
\node at (0,0) {\includegraphics[width=0.5\textwidth]{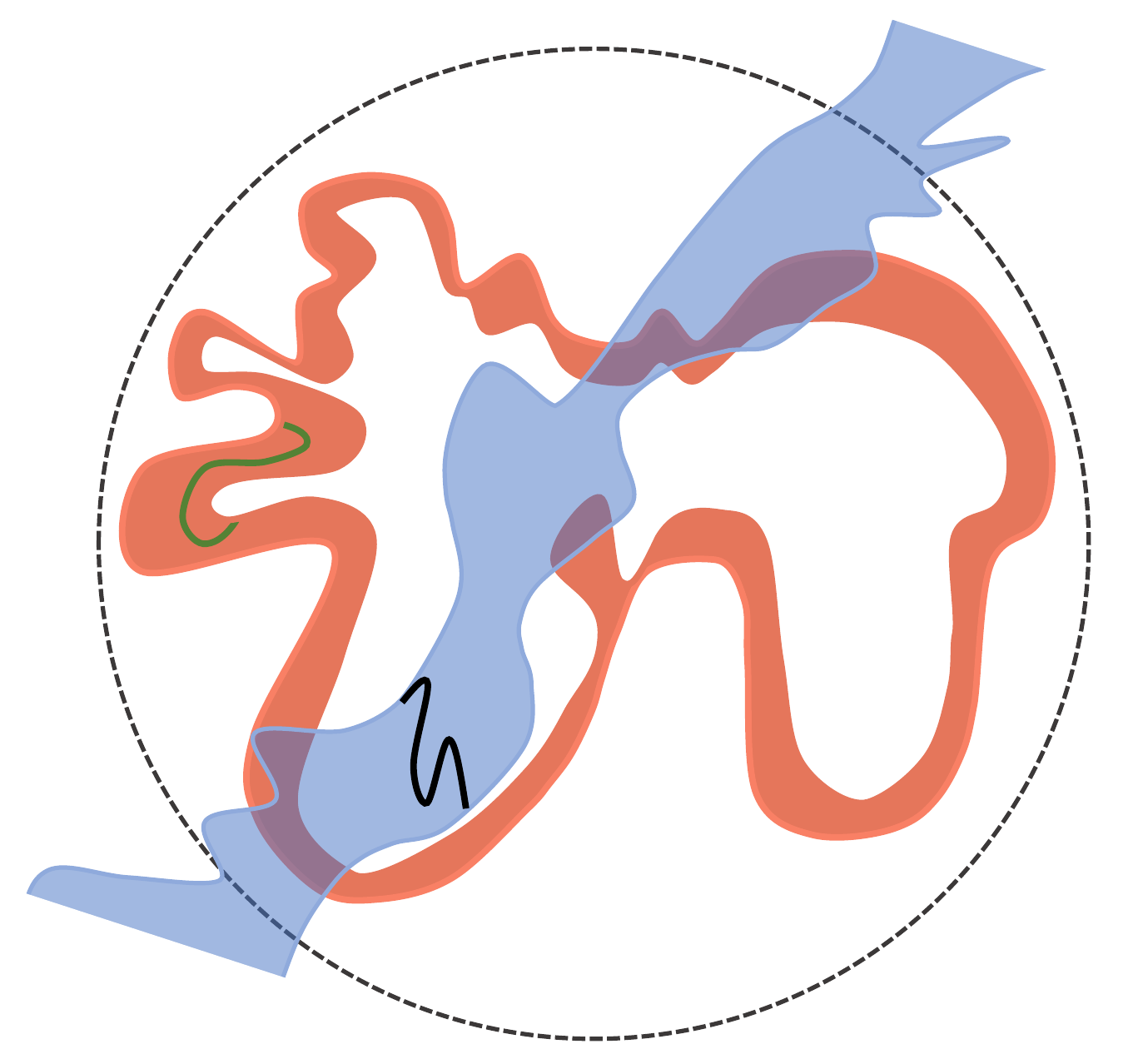}}; 
\node at (1.8,1.2) {$\mathcal{A}$};
\node at (1.1,1.7) {$\mathcal{S}$};
\node at (-0.4,-1) {$\gamma_1$};
\node at (-2.15,0) {$\gamma_2$};
\node at (-2,2) {$B$};
\end{tikzpicture}  
\caption{Linkage condition. The figure represents an example of a topological annulus~(red) linked with a topological strip~(blue) through the existence of a bridge~(black) and a ray~(green).}\label{fig-link}
\end{figure}
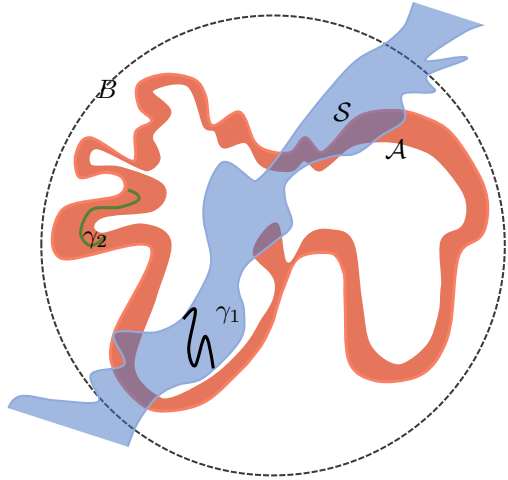

Notice that Definition~\ref{def-link} involves only the geometry inside a topological ball $B$.
Therefore it could include the case when the strip $\mathcal{S}$ is the intersection of an annulus $\mathcal{A}_2$ with the ball $B$.
In this manner we are generalizing the definition of the linkage between two annuli $\mathcal{A}_1,\,\mathcal{A}_2$ given in \cite[Definition~3.2]{PVZ-18}. 
In the following proposition we also recover some of the properties collected in \cite[Proposition~3.1]{PVZ-18} for the linkage of two annuli.

From the third requirement of Definition~\ref{def-link} it follows that the set $B\setminus\mathcal{S}$ has two connected components that will be denoted  $B^+$ and $B^-$.

\begin{proposition}
If the topological strip $\mathcal{S}$ is linked with the topological annulus $\mathcal{A}$, then there exists a topological ball $B$ containing $\mathcal{A}$, a bridge $\gamma_3$ in $\mathcal{S}$ and a ray $\gamma_4$ in $\mathcal{A}$ such that $\gamma_{3}\subset B\setminus in(\partial^e \mathcal{A})$, and denoting by $B^+$ the component of  $B\setminus\mathcal{S}$ that contains  $\gamma_2\subset B^+$, then $\gamma_4\subset B^-$.
\end{proposition}
\begin{proof}
First of all we observe that the existence of a bridge $\gamma_3\subset B\setminus in(\partial^e \mathcal{A})$ follows immediately from  Definition~\ref{def-link}. Indeed, we can choose $\gamma_3$ between one of the two components of $(\mathcal{S}\setminus\gamma_1)\cap\partial B$ and one of the bridges in $(\mathcal{S}\setminus\gamma_1)\cap\partial B$.
 
 The proof of the existence of the ray $\gamma_4$ is entirely analogous to that of \cite[Proposition~3.1]{PVZ-18} and is omitted.
\end{proof}

In the sequel, we deal with the study of the dynamics in a strip~$\mathcal{S}$ and in an annulus~$\mathcal{A}$. If they are linked, then there exist two disjoint topological rectangular regions $\mathcal{R}_1\subset \mathcal{A}\cap\mathcal{S}\cap{B}$ and $\mathcal{R}_2\subset \mathcal{A}\cap\mathcal{S}\cap{B}$.

Firstly, we consider the following continuous map 
\begin{equation}\label{eq-s}
\phi_{\mathcal{S}}\colon\mathcal{S}\to\mathcal{S}.
\end{equation}
Without loss of generality, we can assume that $\mathcal{R}_1,\,\mathcal{R}_2$ are homeomorphic to $R_1=[-2,-1]\times[-1,1]$ and $R_2=[1,2]\times[-1,1]$, respectively. We suppose that the map $\phi_{\mathcal{S}}$ in \eqref{eq-s} admits a lift $\widetilde{\phi}_{\mathcal{S}}$ to the covering space $\mathopen{[}a,b\mathclose{]}\times\mathopen{[}-1,1\mathclose{]}$, with $a<-2$ and $b>2$, defined as
\begin{equation}\label{eq-s-lift}
\widetilde{\phi}_{\mathcal{S}}\colon (x,y)\mapsto(x+\Xi(x,y),\zeta(x,y))
\end{equation}
where $\zeta,\,\Xi$ are continuous functions.

\begin{definition}[Strip boundary invariance condition]\label{def-boundary-strip} 
The condition holds for the map $\phi_{\mathcal{S}}$ if the second coordinate of its lift $\widetilde{\phi}_{\mathcal{S}}$ satisfies $\zeta(x,-1)\equiv -1$ and $\zeta(x,1)\equiv 1$.
\end{definition}

\begin{figure}[htb]
\centering
\begin{subfigure}{.7\textwidth}
\centering
\includegraphics[width=.7\linewidth]{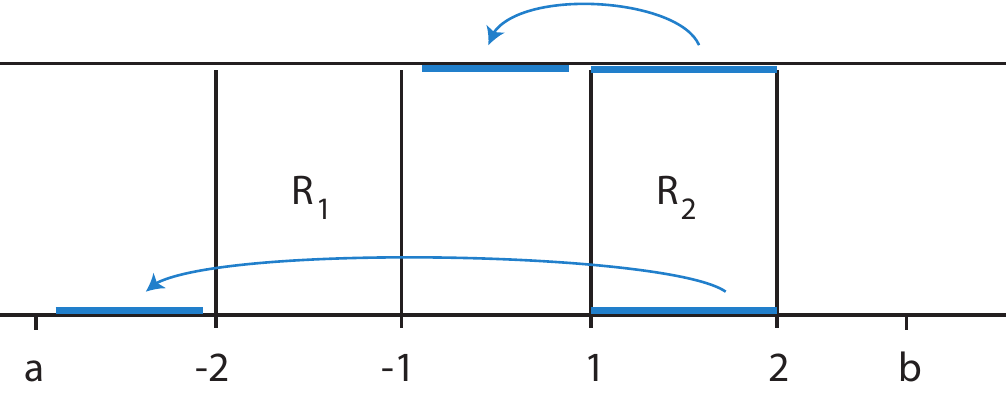}
\caption{Image of $[1,2]\times\{-1\}$ and $[1,2]\times\{1\}$ under a twist condition with respect to the rectangle $R_1$.}
  \label{fig-strip1}
\end{subfigure}\\\vspace{0.3cm}
\begin{subfigure}{.7\textwidth}
\centering
\vspace{0.8em}
\includegraphics[width=.7\linewidth]{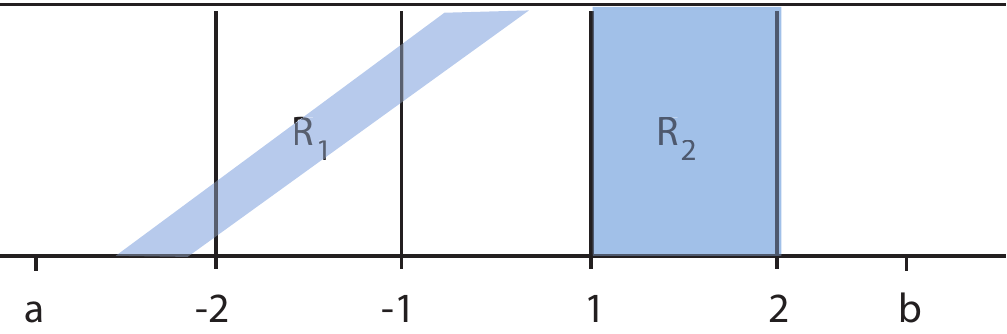}
\caption{Image of the rectangle $R_2$ under a twist condition with respect to the rectangle $R_1$.}
  \label{fig-strip2}
\end{subfigure}%
\caption{Example of strip twist condition.}\label{fig-twist-strip}
\end{figure}

\begin{definition}[Strip twist condition]\label{def-strip} 
The condition holds with respect to $R_1$ for $x\in[1,2]$ if either 
\[
\Xi(x, -1)\leq -4\quad\text{and}\quad\Xi(x,1)\geq -2,
\]
or
\[
\Xi(x, -1)\geq -2\quad\text{and}\quad\Xi(x,1)\leq -4.
\]
The condition holds with respect to $R_2$ for $x\in[-2,-1]$ if either 
\[
\Xi(x,-1)\leq 2\quad\text{and}\quad\Xi(x,1)\geq 4,
\]
or
\[
\Xi(x, -1)\geq 4\quad\text{and}\quad\Xi(x,1)\leq 2.
\]
\end{definition}

Secondly, we consider the following continuous map 
\begin{equation}\label{eq-a}
\phi_{\mathcal{A}}\colon\mathcal{A}\to\mathcal{A}.
\end{equation}
We suppose that the map $\phi_{\mathcal{A}}$ in \eqref{eq-a} admits a lift $\widetilde{\phi}_{\mathcal{A}}$ to the covering space $\mathbb{R}\times\mathopen{[}-1,1\mathclose{]}$ defined as
\begin{equation}\label{eq-a-lift}
\widetilde{\phi}_{\mathcal{A}}\colon (\theta,\rho)\mapsto(\theta+ \Theta(\theta,\rho),\omega(\theta,\rho)),
\end{equation}
where $\theta,\,\rho$ are generalized polar coordinates, and $\Theta,\,\omega$ are continuous functions $1$-periodic in the $\theta$-variable. Without loss of generality, we can assume that $\mathcal{R}_1$ and $\mathcal{R}_2$ are represented in the covering by $R_1=[2k,2k+\frac{1}{2}]\times[-1,1]$ and $R_2=[2k+1,2k+\frac{3}{2}]\times[-1,1]$, respectively.

\begin{figure}[htb]
\centering
\begin{subfigure}{.7\textwidth}
\centering
\includegraphics[width=1\linewidth]{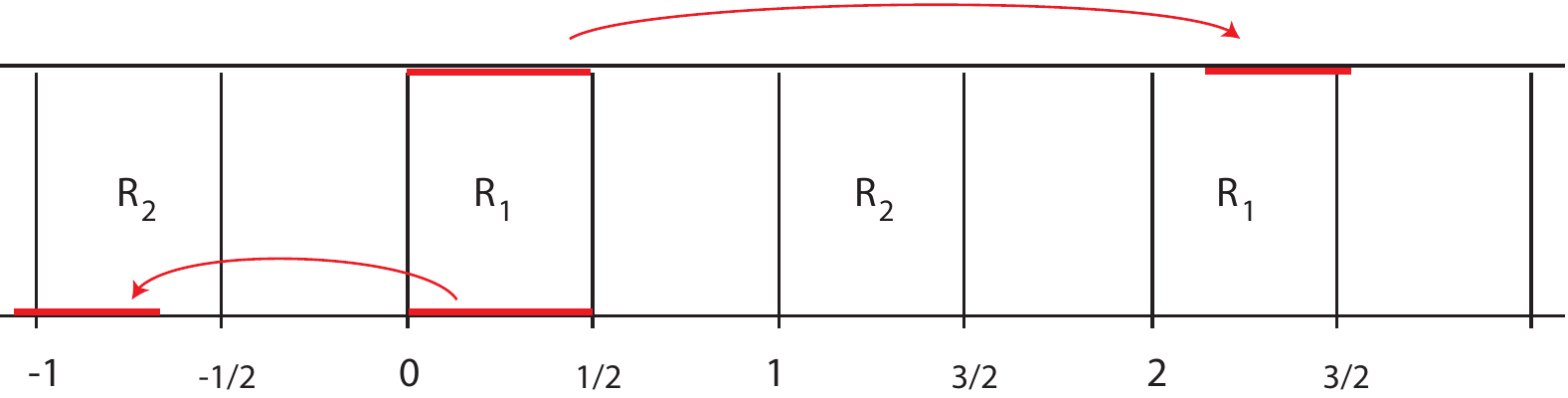}
\caption{Image of $[1,1/2]\times\{-1\}$ and $[1,1/2]\times\{1\}$ under a twist condition with respect to the rectangle $R_1$.}
  \label{fig-ann1}
\end{subfigure}\\\vspace{0.5cm}
\begin{subfigure}{.7\textwidth}
\centering
\includegraphics[width=1\linewidth]{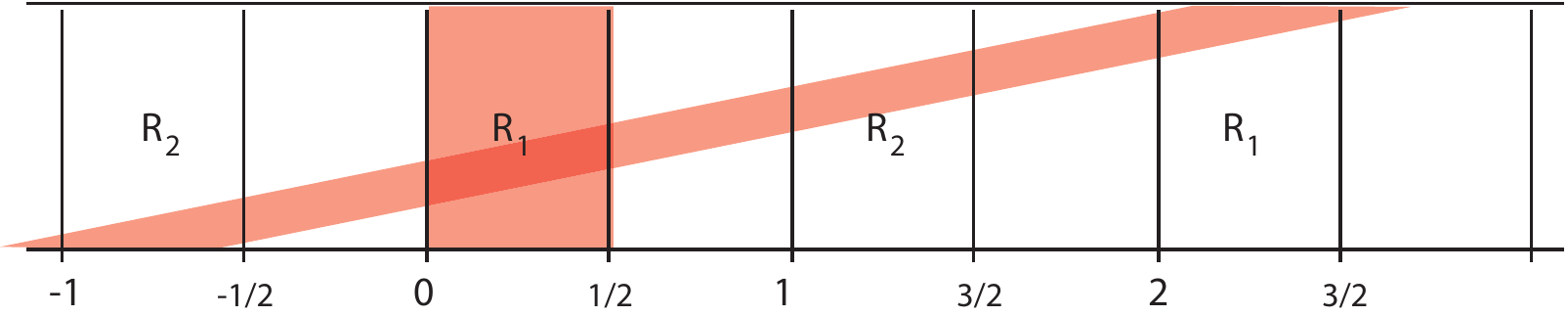}
\caption{Image of the rectangle $R_1$ under a twist condition it goes across a copy of $R_2$. Here $j_{-1}=j_{1}=0$.}
  \label{fig-ann2}
\end{subfigure}%
\caption{Example of an annular twist condition.}\label{fig-twist-ann}
\end{figure}

\begin{definition}[Annular boundary invariance condition]\label{def-boundary-annulus} 
The condition holds for the map $\phi_{\mathcal{A}}$ if the second coordinate of its lift $\widetilde{\phi}_{\mathcal{A}}$ satisfies $\omega(\theta,-1)\equiv -1$ and $\omega(\theta,1)\equiv 1$.
\end{definition}

\begin{definition}[Annular twist condition] \label{def-annulus}
There exist integers $j_{-1}$ and $j_{1}$ such that the condition holds with respect to $R_1$ for $\theta\in[0,1/2]$ if either 
\[
\Theta(\theta, -1)\leq 2j_{-1}+\tfrac{1}{2} \quad\text{and}\quad \Theta(\theta, 1)\geq 2j_{1}+\tfrac{3}{2}, \text{ with } j_{1}+1- j_{-1}>0
\]
or
\[
\Theta(\theta, -1)\geq 2j_{-1}+\tfrac{3}{2} \quad\text{and}\quad \Theta(\theta, 1)\leq 2j_{1}+\tfrac{1}{2}, \text{ with } j_{-1}+1- j_{1}>0
\]
hold.
\end{definition}
We notice that when the annular twist condition holds with respect to $R_1$ then the rectangle $R_1$ is stretched across $R_2$ a number of times which is given by $|j_{-1}-j_{1}|+1$.

\begin{theorem}\label{th-1}
Let $\mathcal{A}$ be a topological annulus linked with a topological strip $\mathcal{S}$. Let $\mathcal{R}_i$ for $i = 1,2$ be two disjoint oriented topological rectangles given through the linkage. Let $\phi_{\mathcal{A}}\colon\mathcal{A}\to\mathcal{A}$ and $\phi_{\mathcal{S}}\colon\mathcal{S}\to\mathcal{S}$, be two continuous maps that satisfy the boundary invariance conditions, and the twist
conditions. Then,
\[
\phi_{\mathcal{A}}\circ\phi_{\mathcal{S}}\colon\widehat{\mathcal{R}}_j\mathrel{\Bumpeq\!\!\!\!\!\!\longrightarrow^{m-1}}\widehat{\mathcal{R}}_j
\quad\text{ and }\quad
\phi_{\mathcal{S}}\circ\phi_{\mathcal{A}}\colon\widehat{\mathcal{R}}_{j+1}\mathrel{\Bumpeq\!\!\!\!\!\!\longrightarrow^{m-1}}\widehat{\mathcal{R}}_{j+1}
\]
for some $j\pmod{2}$ with $m=|j_{-1}-j_{1}|+1$.
\end{theorem}

We notice that \cite[Theorem~3.1]{PVZ-18} becomes a corollary of Theorem~\ref{th-1}. For the proof we use the following lemma.

\begin{lemma}\label{lemma-square}
Consider
\begin{equation}
K_{\ell}=\tilde{\phi}_{\mathcal{A}}\left([2\ell+1,2\ell+3/2]\times[-1,1] \right)\cap R_{1,0}, \quad \ell\in\mathbb{Z}
\end{equation}
where $R_{1,0}=[0,1/2]\times[-1,1]$. If $\phi_{\mathcal{A}}$ satisfies the annular twist condition then at least $m-1$ of the $K_{\ell}$ are non empty with $m=|j_1-j_{-1}|+1$.
\end{lemma}

\begin{proof}
We will prove the lemma in the case of the first annular strip condition, the proof for the second condition being similar.

Let $\theta_0\in[0,1/2]$ be fixed. The vertical segment $(\theta_0,\rho)$, $\rho\in[-1,1]$ is mapped by $\tilde{\phi}_{\mathcal{A}}$ in to a curve. Its end points satisfy 
\[\begin{split}
&\tilde{\phi}_{\mathcal{A}}(\theta_0,-1)=(\theta_{-1},-1) \text{ where }\theta_{-1}\leq\theta_0+2j_{-1}+\tfrac{1}{2}, \\
&\tilde{\phi}_{\mathcal{A}}(\theta_0,1)=(\theta_{1},1) \text{ where }\theta_{1}\geq\theta_0+2j_{-1}+\tfrac{1}{2}+2m-1.
\end{split} 
\]
Hence, $|\theta_{-1}-\theta_1|\geq 2m-1|$ and $K_{\ell}\not=\emptyset$ for $\ell=j_{-1},\dots,j_{-1}+m-1$.
\end{proof}

\begin{proof}[Proof of Theorem~\ref{th-1}.]
First of all without loss of generality we assume that $\phi_{\mathcal{S}}$ maps $\mathcal{R}_2$ across $\mathcal{R}_1$ thanks to the strip twist condition. 
Hence we prove that $\phi_{\mathcal{S}}\circ\phi_{\mathcal{A}}\colon\widehat{\mathcal{R}}_1\mathrel{\Bumpeq\!\!\!\!\!\!\longrightarrow^{m}}\widehat{\mathcal{R}}_1$. The other situations are just an adaptation of this proof. 

We want to find disjoint compact subsets $\mathcal{K}_{1},\dots,\mathcal{K}_{m-1}\subset\mathcal{R}_{1}$ such that for any continuous path $\gamma$ across $\mathcal{R}_{1}$ with $\gamma(0)$, $\gamma(1)$ in different components of $\partial\mathcal{R}_{1}$, the restriction $\left.\phi_{\mathcal{A}}(\gamma(t))\right|_{\mathcal{K}_{\ell}}$ goes across $\mathcal{R}_{2}$. In order to do this we work on the covering space, where the $\mathcal{K}_{\ell}$ will be represented by the $K_{\ell}$ of Lemma~\ref{lemma-square}. The $\mathcal{K}_{\ell}$ are pairwise disjoint because the $K_{\ell}$ lie in a single representative $R_{1,0}$ of $\mathcal{R}_{1}$.

The arguments used in the proof of Lemma~\ref{lemma-square} ensure that the curve $\tilde{\gamma}(t)$ in the covering, satisfying $\tilde{\gamma}(0)=(\theta_0,-1)$, and $\tilde{\gamma}(1)=(\theta_1,1)$ with $\theta_0,\,\theta_1\in[0,1/2]$ goes across all the $K_{\ell}$, and that the restriction of $\tilde{\gamma}$ to each $K_{\ell}$ goes across some copy, $[2\ell+1,2\ell+3/2]\times[-1,1]$, of $R_2$.
\end{proof}

\section{Application to codimension $1$ reversible vector fields}\label{section-4} 
To detect chaotic dynamics, we apply the topological results of the previous section to some periodically forced reversible ODEs. In particular, we consider a $T$-periodic step-wise forcing term $p(t)$ that switches between two different values as follows
\begin{equation}
p(t):=
\begin{cases}
\lambda_1&\text{for }t\in[0,\tau_1),\\
\lambda_2&\text{for }t\in[\tau_1,\tau_1+\tau_2),
\end{cases}
\end{equation} 
where $\lambda_1\not=\lambda_2$ and $0<\tau_1<\tau_2<T$ with $\tau_1+\tau_2=T$. 
We investigate the $T$-periodic problem associated with the system
\begin{equation}\label{eq-4.1}
\begin{cases}
\dot{x}= y f(x,y^2), \\
\dot{y}= g(x,y^2)+p(t),
\end{cases}
\end{equation}
where $f$ and $g$ are smooth functions that identify the normal forms of codimension $1$ reversible systems introduced in~\cite{Te-97}.

Our goal is to prove the existence of chaotic dynamics for system~\eqref{eq-4.1}. First, we look at the flow of the vector field $X(x,y)$ associated with~\eqref{eq-4.1} which is given by the unique solution $(x(t),y(t))=\varphi(t,x_0,y_0)$ of $\dot{X}=X(x,y)$ satisfying $x(0)=x_0$ and $y(0)=y_0$. We study the Poincar\'e map $\Phi\colon\mathbb{R}^2\to\mathbb{R}^2$ defined by $\Phi(x_0,y_0)=\varphi(T,x_0,y_0)$ for every point $(x_0,y_0)\in\mathbb{R}^2$.  
Second, we notice that the full dynamics of the problem can be broken into two sub-systems
\begin{equation}\label{eq-l1}
\begin{cases}
\dot{x}= y f(x,y^2), \\
\dot{y}=g(x,y^2)+\lambda_1,
\end{cases}
\end{equation} 
and 
\begin{equation}\label{eq-l2}
\begin{cases}
\dot{x}= y f(x,y^2), \\
\dot{y}=g(x,y^2)+\lambda_2.
\end{cases}
\end{equation} 
Hence, we have that the Poincar\'e map $\Phi$ may be decomposed as $\Phi=\Phi_{\lambda_2}\circ\Phi_{\lambda_1}$, where, for any $(x_0,y_0)\in\mathbb{R}^2$, $\Phi_{\lambda_1}(x_0,y_0)=\varphi_{\lambda_1}(\tau_1,x_0,y_0)$ and $\Phi_{\lambda_2}(x_0,y_0)=\varphi_{\lambda_2}(\tau_2,x_0,y_0)$ are the Poincar\'e maps associated with~\eqref{eq-l1} and \eqref{eq-l2}, respectively.
We outline here  the structure of the proof for the saddle case, done by applying Theorem~\ref{th-1}.

\begin{enumerate}[label=\arabic*)]
\item Locate a flow invariant line $\Gamma_{1,*}$ for, say  $\lambda_1$ and a closed flow invariant line $\Gamma_{2,*}$ for  $\lambda_2$, making sure they intersect in at least two points.   
Then $\Gamma_{2,*}$ is going to be $\partial^e\mathcal{A}$ and  $\Gamma_{1,*}$ will be of one component of 
$\partial\mathcal{S}$.
\item Take $\tau_1$ to be the time it takes for $\varphi_{\lambda_1} $ to move one intersection point to the next one.
\item Look at a curve $\gamma_1$ ending at the  first intersection point as a candidate for a bridge and make sure $\Phi_{\lambda_1}$ maps it to $in\left(\Gamma_{2,*}\right)$. Take $P$ to be the other end point of $\gamma_1$.
\item Take the $\varphi_{\lambda_1}$ trajectory through $P$ to be the  other component of $\partial\mathcal{S}$ and take
the (closed) $\varphi_{\lambda_2}$ trajectory through $P$ to be $\partial^i\mathcal{A}$.
This ensures that the strip twist condition (Definition~\ref{def-strip}) holds.
\item Obtain the time $\tau_2$ for the annular-strip condition (Definition~\ref{def-annulus}).
\end{enumerate}
In this way we can prove that the dynamics of~\eqref{eq-4.1} is semiconjugate to a shift in a finite alphabet.

\subsection{Saddle case.} We assume that system~\eqref{eq-4.1} has a saddle structure by considering 
\begin{equation}\label{eq-saddle}
\begin{cases}
\dot{x}= x y, \\
\dot{y}= x-y^2+p(t).
\end{cases}
\end{equation}
Depending on $p(t)$, the phase-portrait of system~\eqref{eq-saddle} switches between different configurations as described in Section~\ref{section-2}.

 \begin{figure}[h]
\centering
\includegraphics[height=5cm]{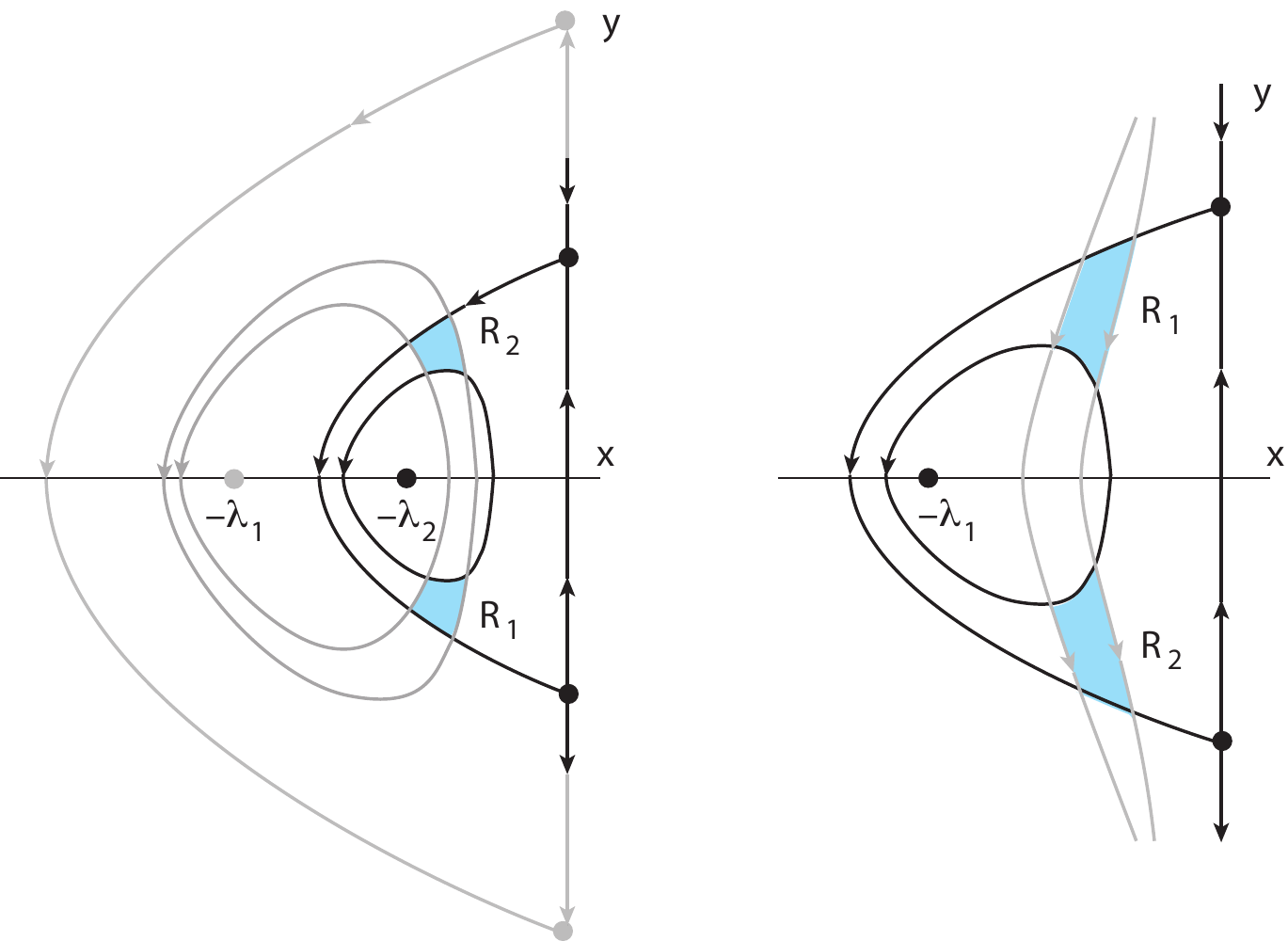} 
\caption{Construction of the annulus $\mathcal{A}$  and the strip $\mathcal{S}$ in the saddle case. Left: $0<\lambda_2<\lambda_1$; right: $\lambda_2\le 0<\lambda_1$.}\label{figSaddleRectangle}
\end{figure}

\begin{theorem}\label{th-saddle}
Let $\Phi$ be the Poincar\'e map associated with system~\eqref{eq-saddle}. Then for each $\lambda_1>0$ and each $\lambda_2$ with $\lambda_1>\lambda_2$ and for an open set of values of $\tau_1$ and $\tau_2$ the map
$\Phi$ induces chaotic dynamics on $m$ symbols, for some $m\ge 2$.
\end{theorem}

\begin{proof}
First of all we notice that the following two cases can occur: $\lambda_1>\lambda_2>0$ or $\lambda_1>0\geq\lambda_2$. 

Let us suppose that $\lambda_1$ and $\lambda_2$ are two fixed positive values satisfying the first case. Then for both systems \eqref{eq-l1} and \eqref{eq-l2} there exist three equilibria. In particular, there exists a heteroclinic cycle around the center $(-\lambda_i,0)$ which joins the two saddles $(0,-\sqrt{\lambda_i})$ and $(0,\sqrt{\lambda_i})$, for $i=1,\,2$.

Let $(x^*,0)$ be the point where the heteroclinic cycle of system \eqref{eq-l2} crosses the negative part of the $x$-axis. Then two configurations are possible: $-\lambda_1<x^*<-\lambda_2$ or $x^*<-\lambda_1$.  It will be not restrictive to consider the first configuration since the other situation can be treated similarly. We proceed with the construction of an annulus $\mathcal{A}$ and a strip $\mathcal{S}$ which satisfy the topological conditions required to apply Theorem~\ref{th-1}. 

For any $(x,y)\in\mathbb{R}^2$, we call $\Gamma_{1}{(x,y)}$ and $\Gamma_{2}{(x,y)}$ the trajectories through the point $(x,y)$ of system \eqref{eq-l1} and \eqref{eq-l2}, respectively. Let $\Gamma_{2}{(x^*,0)}$ be the heteroclinic trajectory through $(x^*,0)$, then we define the outer component of $\partial\mathcal{A}$ as $\partial^{e}\mathcal{A}:=\Gamma_{2}{(x^*,0)}\cup \{(0,-\sqrt{\lambda_2})\}\cup\Gamma_{2}{(0,0)}\cup\{(0,\sqrt{\lambda_2})\}$. Let $\alpha<0$ with $-\lambda_2<\alpha$ be any number so the trajectory $\Gamma_{1}{(\alpha,0)}$ through $(\alpha,0)$ will cross the heteroclinic connection $\Gamma_{2}{(x^*,0)}$. We take $\Gamma_{1}{(\alpha,0)}\cap\{x^*\leq x\leq 0\}$ to be one of the components of $\partial\mathcal{S}$, and we construct the other two boundary pieces of the annulus and the strip so as to satisfy the linkage condition and the twist conditions. 

Let $\tau_1$ be the minimum positive time such that, if $r(t)$ is a solution of \eqref{eq-l1} through $(\alpha,0)$ with $r(0)\in\Gamma_{2}{(x^*,0)}\cap\{y<0\}$, then $r(\tau_1)\in\Gamma_{2}{(x^*,0)}\cap\{y>0\}$. For any point $(x,y)\in\Gamma_{2}{(x^*,0)}\cap\{y<0\}$ close to $r(0)$ the points $\varphi_{\lambda_1}(\tau_1,x,y)$ form a curve through $r(\tau_1)$. Generically this curve goes across $\Gamma_{2}{(x^*,0)}$ (otherwise, make a small change in $\alpha$). Suppose that the curve is below $\Gamma_{2}{(x^*,0)}$ to the left of $r(\tau_1)$ (otherwise the arguments are similar). 
Take $\beta<0$ with $-\lambda_2<\beta<\alpha<0$ such that the points in the trajectory $\Gamma_{1}{(\beta,0)}$ of system \eqref{eq-l1} through $(\beta,0)$ satisfy the condition on the curve. Then we take the other component of $\partial\mathcal{S}$ as $\Gamma_{1}{(\beta,0)}\cap\{x^*\leq x\leq 0\}$. It remains to obtain the inner component of $\partial\mathcal{A}$.

Let $\Pi\colon\mathbb{R}^2\to\mathbb{R}^2$ be the projection on the second component, namely $\Pi(x,y)=y$. For any $(x,y)\in\mathbb{R}^2$ let $\psi(x,y)=\Pi(\varphi_{\lambda_1}(\tau_1,x,y))$ and let $\overline{\psi}(x,y)=\psi(x,y)+\Pi(x,y)$, so $\overline{\psi}(x,y)$ compares the height of $\varphi_{\lambda_1}(\tau_1,x,y)$ to that of the symmetric point of $(x,y)$.
 
 Let $q(t)$ be the solution of \eqref{eq-l1} through $(\beta,0)$ with $q(0)\in\Gamma_{2}{(x^*,0)}\cap\{y<0\}$. Then $\overline{\psi}(q(0))<0$. Also there exists a $\sigma>0$ such that $q(\sigma)\in\Gamma_{2}{(x^*,0)}\cap\{y>0\}$. By construction, $\overline{\psi}(q(\sigma))>0$. Therefore, there exists $\widehat{\sigma}\in(0,\sigma)$ such that $\overline{\psi}(q(\widehat{\sigma}))=0$. This means that $\varphi_{\lambda_1}(\tau_1,q(\tau_1))$ is symmetric to $q(\tau_1)$. The trajectory $\Gamma_{2}(q(\widehat{\sigma}))$ will go through both $q(\widehat{\sigma})$ and $\varphi_{\lambda_1}(\tau_1,q(\widehat{\sigma}))$. We define the inner component of $\partial\mathcal{A}$ as $\partial^{i}\mathcal{A}:=\Gamma_{2}(q(\widehat{\sigma}))$.
 
In this manner, the topological annulus $\mathcal{A}$ and the topological strip $\mathcal{S}$ are linked by construction (see Figure~\ref{figSaddleRectangle}). The linkage condition gives two symmetric topological rectangles $\mathcal{R}_1$ and $\mathcal{R}_2$ (in the lower and upper half-plane, respectively) that satisfy the twist conditions. Indeed, a strip-twist condition holds for $\Phi_{\lambda_1}\colon\mathcal{S}\to\mathcal{S}$ because the rectangle $\mathcal{R}_1\subset\mathcal{A}\cap\mathcal{S}\cap\{y<0\}$ is stretched across $\mathcal{R}_2\subset\mathcal{A}\cap\mathcal{S}\cap\{y>0\}$.
Since $\Gamma_{2}(x^*,0)$ is a heteroclinic connection then for every $m\geq2$ there exists $\tau_2$ large enough such that  an annulus-twist condition also holds for $\Phi_{\lambda_2}\colon\mathcal{A}\to\mathcal{A}$ because $\mathcal{R}_2$ is stretched across $\mathcal{R}_1$ $m$-times (depending on $\tau_2$).
 The result follows by an application of Theorem~\ref{th-1} to the Poincar\'e map $\Phi=\Phi_{\lambda_2}\circ\Phi_{\lambda_1}$. This concludes the first case.

The proof above holds for a fixed value of $\tau_1$ and for sufficiently large $\tau_2$.
However, we may obtain the result for $\tau_1$ in an open interval by taking different values of $\alpha$. 

\medbreak

The arguments above yield a proof for the case $\lambda_1>0\geq\lambda_2$, we just indicate where it needs to be adapted.
The outer component of $\partial\mathcal{A}$ may be taken as $\partial^{e}\mathcal{A}:=\Gamma_{1}{(x^*,0)}\cup \{(0,-\sqrt{\lambda_1})\}\cup\Gamma_{1}{(0,0)}\cup\{(0,\sqrt{\lambda_1})\}$, where $\Gamma_{1}{(x^*,0)}$ is the heteroclinic trajectory of $\varphi_{\lambda_1}$ going through $(x^*,0)$. One of the components of $\partial\mathcal{S}$ will be $\Gamma_2(\alpha,0)$ with $-\lambda_1<\alpha<0$.

Then take $\tau_2$ to be the least positive time to go from  $\Gamma_2(\alpha,0)\cap\Gamma_{1}{(x^*,0)}\cap\{y>0\}$ to $\Gamma_{1}{(x^*,0)}\cap\{y<0\}$.
Apply the arguments above to obtain the other component of $\partial\mathcal{S}$ as a $\varphi_{\lambda_2}$ trajectory that starting at $\Gamma_{1}{(x^*,0)}\cap\{y>0\}$ arrives above $\Gamma_{1}{(x^*,0)}\cap\{y<0\}$ in time $\tau_2$.
Then find  a point $q$  in this trajectory  and in the upper half-plane, such that $\Phi_{\lambda_2}$ maps $q$ to its symmetric $h(q)$.
Take $\partial^{i}\mathcal{A}:=\Gamma_2(q)$ to complete the construction.
\end{proof}

In the case when both  $\lambda_1$ and $\lambda_2$ are negative there are no annular invariant regions, so the results cannot be applied. 
Moreover, in this case there are no non-trivial  periodic orbits, so we do not expect  periodic forcing to yield chaos.
The same holds for the cusp case below, when both  $\lambda_1$ and $\lambda_2$ are positive.

\subsection{Cusp case.}  When system~\eqref{eq-4.1} has the following form 
\begin{equation}\label{eq-cusp}
\begin{cases}
\dot{x}= y, \\
\dot{y}= x^2+p(t).
\end{cases}
\end{equation}
then its phase-portrait is of cusp type. We notice that system~\eqref{eq-cusp} has also a Hamiltonian structure, and at this juncture, when  $\lambda_1<0$ and $\lambda_2\leq0$ the geometry is similar to the one investigated in \cite{So-sub,SoZa-17}. Hence, we expect that chaotic dynamics occurs for $\tau_1$ and $\tau_2$ large enough. 
For Theorem~\ref{th-saddle} we have used a heteroclinic connection to obtain an annulus twist condition. Here the existing homoclinic connection may be used for the same purpose and,
by applying the procedure exploited for Theorem~\ref{th-saddle}, we can prove what follows.

\begin{theorem}\label{th-cusp}
Let $\Phi$ be the Poincar\'e map associated with system~\eqref{eq-cusp}. Then for each $\lambda_1\leq0$ and each $\lambda_2$ with $\lambda_1<\lambda_2$ and for an open set of values of $\tau_1$ and $\tau_2$ the map $\Phi$ induces chaotic dynamics on $m\ge 2$ symbols.
\end{theorem}

The case $\lambda_1<\lambda_2<0$ of Theorem~\ref{th-cusp} may also be obtained as a corollary to \cite[Theorem~4.1]{MRZ-10}. Our methods provide an alternative proof and extend the result to the case $\lambda_1<0$, $\lambda_2>0$.
In the latter case there is no invariant annulus for $\lambda_2>0$ and for the proof we need to use a strip condition.


\bigskip\Addresses
\end{document}